\documentclass{amsart}

\usepackage[utf8]{inputenc}
\usepackage{amsmath}
\usepackage{amsthm,ifthen, amsfonts, amssymb,
srcltx, amsopn, color, enumerate}
\usepackage[a4paper, centering, total={6in, 8in}]{geometry}
 \usepackage{mathabx}
\usepackage{amssymb}
\usepackage[mathscr]{euscript}
\usepackage{enumerate}
\usepackage{graphicx}
\usepackage{tikz}
\usepackage{tikz-cd}
\usepackage{color}
\usepackage{multicol}
\usepackage{mathrsfs}
\usepackage[normalem]{ulem}
\usepackage{float}
\usepackage{subcaption}
\usepackage{mathabx}
\usepackage{hieroglf}
\usepackage[T1]{fontenc}
\usepackage[cmtip,arrow]{xy}
\usepackage{pb-diagram, pb-xy}
\usepackage{overpic}
\usepackage{verbatim}

\newsavebox{\commentbox}
{\ifthenelse{\equal{\showcommentsbox}{yes}}%
{\footnotemark
        \begin{lrbox}{\commentbox}
        \begin{minipage}[t]{1.25in}\raggedright\sffamily\tiny
        \footnotemark[\arabic{footnote}]}
{\begin{lrbox}{\commentbox}}}%
{\ifthenelse{\equal{\showcommentsbox}{yes}}%
{\end{minipage}\end{lrbox}\marginpar{\usebox{\commentbox}}}
{\end{lrbox}}}

\definecolor{Green}{RGB}{30, 150, 30}

\usetikzlibrary{arrows}
\usepackage{tikz-cd} 

\newtheorem{thm}{Theorem}[section]
\newtheorem{prop}[thm]{Proposition}
\newtheorem{claim}[thm]{Claim}
\newtheorem{lem}[thm]{Lemma}

\newtheorem{cor}[thm]{Corollary}

\theoremstyle{definition}

\theoremstyle{definition}

\theoremstyle{definition}

\theoremstyle{definition}
\newtheorem{remark}[thm]{Remark}
\newtheorem{question}[thm]{Question}
\theoremstyle{definition}

\theoremstyle{definition}

\theoremstyle{definition}

\theoremstyle{definition}

\newcommand*\Aut{\operatorname{Aut}}

\newcommand{\Out}{\mathrm{Out}}

\newcommand*\Stab{\operatorname{Stab}}

\newcommand*\diam{\operatorname{diam}}

\newcommand{\kernel}{\mathrm{Ker}}

\newcommand{\tsh}[1]{\left\{\kern-.7ex\left\{#1\right\}\kern-.7ex\right\}}

\newcommand{\stab}{\operatorname{Stab}}

\newcommand{\Inn}{\mathrm{Inn}}

\begin{document}
\title[Separable convex-cocompact subgroups]{Some examples of separable convex-cocompact subgroups}

\author{Mark Hagen}
\address{School of Mathematics, University of Bristol, Bristol, UK}
\email{markfhagen@posteo.net}

\author{Alessandro Sisto}
	\address{Maxwell Institute and Department of Mathematics, Heriot-Watt University, Edinburgh, UK}
	\email{a.sisto@hw.ac.uk}
	
\date{\today}

\begin{abstract}
    Reid asked whether all convex-cocompact subgroups of mapping class groups are separable. Using a construction of Manning-Mj-Sageev, we give examples of separable convex-cocompact subgroups that are free of arbitrary finite rank, while prior examples seem to all be virtually cyclic.  
\end{abstract}

\maketitle

\section{Introduction}

A subgroup $H$ of a group $G$ is \emph{separable (in $G$)} if, for all $g\notin H$, there exists a finite group $F$ and 
a homomorphism $\phi:G\to F$ such that $\phi(g)\notin \phi(H)$. Separability arises naturally in topology when, given an 
immersed submanifold $N$ of a manifold $M$, one tries to find a finite cover of $M$ where $N$ embeds; see \cite{Scott}. 
Notably, separability, and more specifically separability of quasiconvex subgroups of certain hyperbolic groups, played a 
crucial role in the resolution of the Virtual Haken and Virtual Fibering Conjectures \cite{Wise,Agol}.

In the context of mapping class groups, the closest analogues of quasiconvex subgroups of hyperbolic groups are 
\emph{convex-cocompact} subgroups, as defined in \cite{FM:conv_cocpt}; see Subsection \ref{subsec:MCG} for the 
characterisation that we will use. In analogy with the case of hyperbolic manifolds, Reid asked whether convex-cocompact 
subgroups of mapping class groups are separable \cite[Question 3.5]{Reid:problems}. As far as we are aware, the only examples 
of convex-cocompact subgroups known to be separable are virtually cyclic, as covered by \cite{LM:separable}. However, it is 
shown in \cite{CHHS} that if "enough" hyperbolic groups are residually finite then all convex-cocompact subgroups are 
separable. (Note that it is known that not all finitely generated subgroups of mapping class groups are separable 
\cite{LM:separable}.)

The main goal of this paper is to provide the first examples of separable convex-cocompact subgroups that are not virtually cyclic, using the criterion in Theorem~\ref{thm:criterion}.

Write $\Sigma_{\mathfrak g}$ to mean the closed connected orientable surface of genus $\mathfrak g$ and 
$MCG(\Sigma_{\mathfrak g})$ its mapping class group.  Given $H\leq MCG(\Sigma_{\mathfrak g})$, denote by $\Gamma_H$ the 
corresponding extension of $H$. This can be seen as the preimage of $H$ under the natural map $MCG(\Sigma_{\mathfrak 
g}-\{p\})\to MCG(\Sigma_{\mathfrak g})$ occurring in the Birman exact sequence (see e.g. \cite[Section 1.2]{FM:conv_cocpt} 
for detailed description of $\Gamma_H$). When $H$ is convex-cocompact, $\Gamma_H$ is hyperbolic \cite{Hamenstadt} (if in 
addition $H$ is free, as it will be in our examples, hyperbolicity of $\Gamma_H$ is also \cite[Theorem 1.3]{FM:conv_cocpt}).

The criterion is:

\begin{thm}\label{thm:criterion}
Let $\mathfrak g\geq 1$ and let $H$ be a torsion-free malnormal convex-cocompact subgroup of $MCG(\Sigma_{\mathfrak g})$. If $\Gamma_H$ is conjugacy separable, then $H$ is separable in $MCG(\Sigma_{\mathfrak g})$.
\end{thm}

To apply the criterion, we need examples of convex-cocompact subgroups $H$ with $\Gamma_H$ conjugacy separable.  These come from work of Manning-Mj-Sageev \cite{MMS} -- who gave examples where $\Gamma_H$ is virtually compact special -- combined with Minasyan--Zalesskii's result on conjugacy separability of hyperbolic virtually special groups \cite{MZ}. So, the examples provided by the following theorem come from \cite{MMS}; we just observe that their construction gives rise to malnormal subgroups in many cases.

\begin{thm}\label{thm:intro_examples}
For every $\mathfrak g\geq 3$ and $n\geq 1$, there exist infinitely many conjugacy classes of malnormal convex-cocompact subgroups $H<MCG(\Sigma_{\mathfrak g})$ such that $H$ is isomorphic to a free group of rank $n$ whose corresponding extension $\Gamma_H$ is virtually compact special.
\end{thm}

Recalling that convex-cocompactness of $H$ implies hyperbolicity of $\Gamma_H$, combining the above theorem with 
\cite[Theorem 1.1]{MZ} and Theorem~\ref{thm:criterion} gives:

\begin{cor}\label{cor:separable_examples}
For every $\mathfrak g\geq 3$ and $n\geq 1$ there exist infinitely many conjugacy classes of separable  convex-cocompact 
subgroups $H$ of $MCG(\Sigma_{\mathfrak g})$, with each $H\cong F_n$.
\end{cor}

We note that Theorem \ref{thm:criterion} does not apply to arbitrary cyclic convex-cocompact subgroups,  even though, in this 
case, $\Gamma_H$ is the fundamental group of a fibred hyperbolic 3-manifold, which is virtually compact special 
\cite{Dufour,Agol}.  Indeed, such cyclic subgroups may not be malnormal, and the maximal elementary subgroups containing them 
may not be torsion-free (in case the maximal elementary subgroup containing the cyclic group is torsion-free, the criterion 
does apply, though).  We wonder:

\begin{question}
To what extent can malnormality and/or torsion-freeness be relaxed in Theorem~\ref{thm:criterion}?
\end{question}


\subsection*{Outline of proofs}
In this paper, we did not pursue the most general versions of our intermediate results, instead leaving questions that we believe to be of independent interest; we highlight the main ones below.

To prove Theorem \ref{thm:criterion}, we use a version of Grossman's criterion for residual finiteness of an outer automorphism group \cite{Grossman} (Grossman's criterion is one way to prove residual finiteness of $MCG(\Sigma_{\mathfrak g})$). This is Proposition \ref{prop:super_grossman} and has two hypotheses. The first one, in the case of convex-cocompact subgroups, is conjugacy separability of the extension group. The second one translates to a question about automorphisms of surface groups. Roughly, given a surface group $G$, a group $K$ of automorphisms of $G$, and automorphism $\phi$ not in $K$, one can ask whether there exists $x\in G$ such that $\phi(x)$ is not conjugate to $k(x)$ for any $k\in K$. In our case, we are interested in $K=\Gamma_H$ for some convex-cocompact subgroup $H$ (where the mapping class group of the punctured surface can be identified with an index 2 subgroup of the automorphism group of $G$).

We check this second condition for suitable $H$ in two steps. First, Proposition \ref{prop:non-pA} reduces the problem above to an analogous problem where $K$ is replaced by finitely many automorphisms. This is done by showing that left cosets of $H$ contain finitely many elements that are not pseudo-Anosov. This is where we use the hypotheses that $H$ is torsion-free and malnormal.




The second step is to construct a suitable $x$, given finitely many automorphisms. We do so specifically for surface groups in Proposition \ref{prop:not-conjugate}, but, inspired by the fact that Grossman's criterion can be applied to acylindrically hyperbolic groups \cite{AMS:pointwise}, we ask:

\begin{question}
Let $G$ be a torsion-free acylindrically hyperbolic group, and let $\phi_1,\dots,\phi_n$ be non-inner automorphisms of $G$. Does there exist $x\in G$ with $x$ and $\phi_i(x)$ non-conjugate for all $i$?
\end{question}

In fact, we do not know the answer to this question even in the special case where $G$ is hyperbolic. One way of answering the question affirmatively would be to show that a "generic" $x$ has the required property, which would be interesting in its own right. That is:

\begin{question}
Let $G$ be a torsion-free acylindrically hyperbolic group, let $\phi$ be a non-inner automorphism of $G$, and let $(w_n)$ be a simple random walk on $G$. Is it true that, with probability going to $1$ as $n$ goes to infinity, $w_n$ is not conjugate to $\phi(w_n)$?
\end{question}

This question is already of interest for $G$ hyperbolic and the techniques of \cite{MaherSisto} seem relevant.

\subsection*{Acknowlegdments}  We are grateful to Michah Sageev for answering a question about MMS subgroups and to Ashot Minasyan for several very helpful comments and corrections.  We also thank the referee for a number of suggestions that improved the paper.

\section{Preliminaries}


\subsection{Curve graphs and convex-cocompactness}\label{subsec:MCG}
Given a finite-type surface $\Sigma$, we denote by $\mathcal C(\Sigma)$ its curve graph. For all finite-type surfaces $\Sigma$, the curve graph $\mathcal C(\Sigma)$ is hyperbolic \cite{MM1} and, in fact, say, $100$-hyperbolic \cite{HPW}, see also \cite{Aougab,Bowditch:uniform,PS:bicorn}. Also, the action of $MCG(\Sigma)$ on $\mathcal C(\Sigma)$ is acylindrical \cite{Bow:tight}.  A finitely generated subgroup of $MCG(\Sigma)$ is \emph{convex-cocompact} if the orbit maps to $\mathcal C(\Sigma)$ are quasi-isometric embeddings~\cite[Theorem 1.3]{KentLeininger}.

\subsection{Elementary subgroups}\label{subsec:elementary}
Throughout the paper, we use the following notation.  If $X$ is a hyperbolic space and $Y$ is a quasiconvex subspace, then the closest-point projection $\pi_Y$ is the bounded set defined by

$$\pi_Y(x)=\{y\in Y:d(x,y)\leq d(x,Y)+1\}.$$

If a group $A$ acts by isometries on $X$, then $\pi_Y$ is $\Stab_A(Y)$--equivariant. Next, we summarise properties of \emph{elementary closures} of loxodromic elements used later, as first considered by \cite{BF:WPD}. 

\begin{prop}\label{prop:E}
Let the group $G$ act acylindrically on the hyperbolic space $X$.  There exists a constant $Q$ such that the following holds.  Let $g\in G$ be a loxodromic element.  Then:
\begin{enumerate}
    \item\label{item:quasiline} There exists a unique maximal elementary subgroup $E(g)$ containing $g$, and a $Q$-quasiconvex $E(g)$--invariant $Q-$quasiline $Y=Y_g$.
    
    \item\label{item:coarse_stab}  For every $D\geq 0$ there exists $L\geq 0$ (depending on $D$ and $g$), such that the following holds.  If $h\in G$ has the property that there exist $x,y\in Y$ with $d(x,y)\geq L$ and $d(x,h\cdot x),d(y,h\cdot y)\leq D$, then $h\in E(g)$.
    
\end{enumerate}
\end{prop}

\begin{proof}
Item~\eqref{item:quasiline} is given by  \cite[Lemma 6.5]{DGO}.  Item \eqref{item:coarse_stab} is \cite[Proposition 6.(2)]{BF:WPD}.
\end{proof}

\subsection{Intersection of conjugates vs coarse intersections of coset orbits}\label{subsec:intersections}
We also need a certain correspondence between the behaviour of orbits in $\mathcal C(\Sigma)$ of cosets of convex-cocompact subgroups and intersections of conjugates. 
The next lemma, which is used in the proof of Proposition~\ref{prop:non-pA} below, follows from arguments in \cite{AMST}. We note that \cite[Proposition 4.1]{AbbottManning:A_QI} is related to the lemma, though neither implies the other.

\begin{lem}\label{lem:??}
Let $\mathfrak g\geq 2$ and let $H$ be a convex-compact subgroup of $MCG(\Sigma_{\mathfrak g})$.  Let $o\in \mathcal C(\Sigma_{\mathfrak g})$. Suppose that for some $x\in MCG(\Sigma_{\mathfrak g})$ and some $C\geq 0$, the diameter of $xH \cdot o\cap N_C(H \cdot o)\subseteq \mathcal C(\Sigma_{\mathfrak g})$ is infinite. Then the subgroup $H\cap xHx^{-1}$ is infinite.
\end{lem}

\begin{proof}
The convex-cocompact subgroup $H$ in the statement, acting on $\mathcal C(\Sigma_{\mathfrak g})$, satisfies an appropriately modified version of the \emph{WPD condition} from \cite[Definition 2.12]{AMST}. More specifically, said definition is stated for the action on a Cayley graph with respect to a possibly infinite generating sets. The equivariant quasi-isometry classes of said Cayley graphs are in natural bijection with geodesic spaces with a cobounded action, so dealing with Cayley graphs or cobounded actions is equivalent.

The proof of \cite[Lemma 5.2]{AMST}, and hence of \cite[Proposition 5.3]{AMST}, only uses the WPD property and hence it goes through in our setting (replacing subgroups/cosets with their orbits in $\mathcal C(\Sigma_{\mathfrak g})$ etc.). In particular, we conclude that $xH \cdot o\cap N_C(H \cdot o)\subseteq \mathcal C(\Sigma_{\mathfrak g})$ lies within finite Hausdorff distance of $(H\cap xHx^{-1})\cdot o$. Since the former has infinite diameter, we have that $H\cap xHx^{-1}$ is infinite, as required.
\end{proof}

\subsection{Concatenating geodesics in hyperbolic spaces} 
Later, we will verify that certain elements are pseudo-Anosov by constructing quasigeodesic axes using the following lemma, which we prove by reducing it to \cite[Lemma 4.2]{HW:irred}.

\begin{lem}\label{lem:concat}
For every $\delta\geq 0$ there exists a constant $C_\delta$ such that the following holds. Let $(\alpha_i)_{i\in \mathbb Z}$ be geodesics in a $\delta$-hyperbolic geodesic space, of length bounded independently of $i$, with the terminal point of $\alpha_i$ being the initial point of $\alpha_{i+1}$. Let $$d_{i,j}=diam (N_{3\delta}(\alpha_i)\cap \alpha_j)$$ and assume that for all even $i$, 
$$\ell(\alpha_i)\geq d_{i,i+1}+d_{i,i+2}+d_{i,i-1}+d_{i,i-2}+C_\delta.$$
Then the concatenation of the $\alpha_i$ is a quasi-geodesic.
\end{lem}

\begin{figure}[h]\label{fig:concat}
\includegraphics[width=0.6\textwidth]{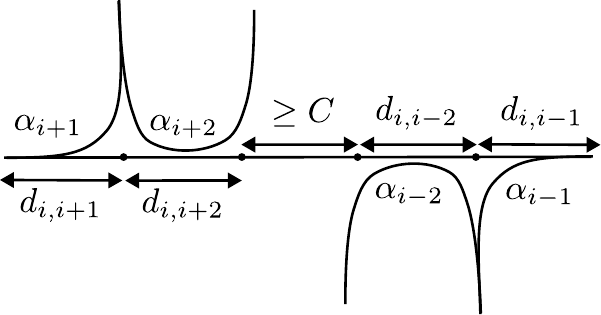}
\caption{An illustration of the hypotheses in Lemma~\ref{lem:concat}.  The long horizontal segment is $\alpha_i$ and the central subsegment is $\alpha_i'$.}
\end{figure}

\begin{proof}
We can assume $\delta>0$ for convenience.  For $i$ even, let $\alpha'_i$ be the subgeodesic of $\alpha_i$ obtained by removing the initial open segment of length $d_{i,i-1}+d_{i,i-2}$ and the final open segment of length $d_{i,i+1}+d_{i,i+2}$. 

For $i$ odd, let $\alpha'_i$ be a geodesic joining the terminal point of $\alpha'_{i-1}$ to the initial point of $\alpha'_{i+1}$.  So, the geodesics $\alpha'_i$ concatenate to form a bi-infinite path.


Considering thin quadrangles, we now show that that \cite[Lemma 4.2]{HW:irred} applies, that is, we show that for $i$ even,
the intersections $$N_{3\delta}(\alpha'_i)\cap \alpha'_{i\pm1}$$ and $$N_{3\delta}(\alpha'_i)\cap \alpha'_{i\pm2}$$ each have diameter at most $10\delta$.


We will do the argument for $\alpha'_{i+1}$ and $\alpha'_{i+2}$, the other arguments being analogous. Denoting by $\alpha''_i$ the terminal segment of $\alpha_i$ that starts at the terminal point of $\alpha'_i$, note that $\alpha'_{i+1}$ is contained in the $2\delta$-neighborhood of $\alpha''_i\cup\alpha_{i+1}\cup\alpha_{i+2}$, and clearly so is $\alpha'_{i+2}$. Hence, it suffices to show that there is no point $p$ on $\alpha'_i$ further than $20\delta$ from the terminal point of $\alpha'_i$ but within $5\delta$ of $\alpha_{i+1}\cup \alpha_{i+2}$ (there is clearly no point $p$ on $\alpha'_i$ further than $20\delta$ from the terminal point and within $5\delta$ of $\alpha''_i$).

\begin{figure}[h]
    \includegraphics[width=0.6\textwidth]{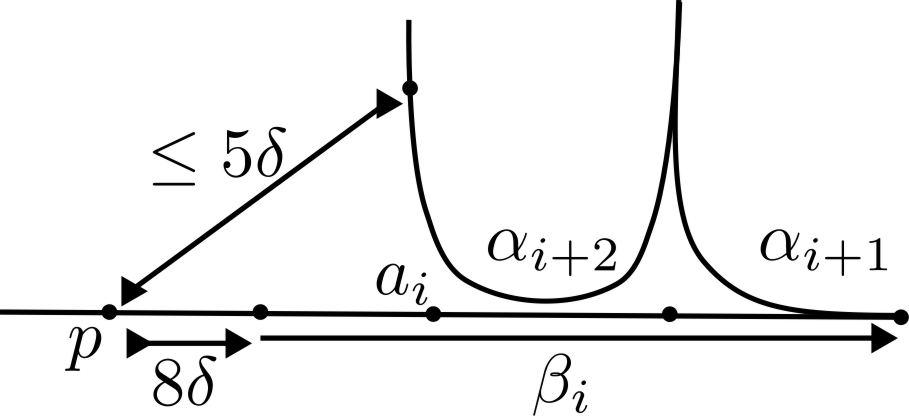}
    \caption{Proof of the nonexistence of $p$ in Lemma~\ref{lem:concat}.  Note that $d(a_i,p)\geq 20\delta$.}
    \label{fig:beta_i}
\end{figure}

Suppose by contradiction that such $p$ exists, and in fact that it lies within $5\delta$ of $\alpha_{i+2}$, the other case being easier.  The following argument is illustrated in Figure~\ref{fig:beta_i}.  Let $a_i$ be the terminal point of $\alpha_i'$, and let $\beta_i$ be the union of the terminal segment of $\alpha_i'$ of length $d(a_i,p)-8\delta$, and $\alpha_i''$. 

Consider the $2\delta$--thin quadrilateral formed by the following four geodesics:
\begin{itemize}
    \item a geodesic $\sigma$ of length at most $5\delta$ from $p$ to a point on $\alpha_{i+2}$;
    \item the subpath of $\alpha_{i+2}$ subtended by the terminal point of $\sigma$ and $\alpha_{i+1}\cap\alpha_{i+2}$;
    \item $\alpha_{i+1}$;
    \item the terminal subpath of $\alpha_i$ from $p$ to the initial point of $\alpha_{i+1}$.
\end{itemize}
By $2\delta$--thinness of the quadrilateral, the fourth path on the list is contained in the union of the $2\delta$--neighbourhoods of the other three paths.  Since the initial point of $\beta_i$ is $8\delta$--far from $p$, and $|\sigma|\leq 5\delta$, we see that no point of $\beta_i$ is $2\delta$--close to $\sigma$, for otherwise it would be $7\delta$--close to $p$, a contradiction.  Hence any point on $\beta_i$ lies $2\delta$-close to either $\alpha_{i+1}$ or $\alpha_{i+2}$.

Note that $\beta_i\cap N_{2\delta}(\alpha_{i+1})$ has diameter at most $d_{i,i+1}+4\delta$.  Indeed, if $p,q\in\beta_i$ are respectively $2\delta$--close to points $p',q'\in\alpha_{i+1}$, then $p',q'$ are in $N_{3\delta}(\alpha_i)\cap\alpha_{i+1}$, so $d(p',q')\leq d_{i,i+1}$, and hence, by the triangle inequality, $d(p,q)\leq d_{i,i+1}+4\delta$.  Similarly, $\beta_i\cap N_{2\delta}(\alpha_{i+2})$ has diameter at most $d_{i,i+2}+4\delta$.  Since these two intersections cover $\beta_i$, we conclude that $|\beta_i|\leq d_{i,i+1}+d_{i,i+2}+8\delta$.  

On the other hand, $\beta_i$ is the concatenation of a terminal segment of $\alpha_i'$ of length at least $12\delta$ (since $p$ is assumed to lie $20\delta$--far from the terminal point of $\alpha_i'$) with the path $\alpha_i''$ (whose length is $d_{i,i+1}+d_{i,i+2}$).  Hence $|\beta_i|>d_{i,i+1}+d_{i,i+2}+8\delta$.  This is a contradiction, so the point $p$ cannot exist.


Hence, for sufficiently large $C_\delta$, \cite[Lemma 4.2]{HW:irred} applies, yielding that the concatenation of the $\alpha'_i$ is a quasi-geodesic. Hence, so is the concatenation of the $\alpha_i$, since all geodesics involved have uniformly bounded length.
\end{proof}

\subsection{Grossman criterion}\label{subsec:grossman}
We now give the generalisation of Grossman's residual finiteness criterion mentioned in the introduction.  Assume that the group $G$ has trivial centre, and let $i:G\to\Inn(G)$ be the natural isomorphism defined by $i(g)(x)=gxg^{-1}$. 

\begin{remark}\label{rem:conjugation_in_extension}
For any $x\in G$ and $\gamma\in\Aut(G)$, 
$$\gamma i(x)\gamma^{-1}=i(\gamma(x)),$$
just by expanding the definition of $i$.
\end{remark}

We also need the usual exact sequence

$$1\to \Inn(G)\to \Aut(G)\stackrel{\Psi}{\longrightarrow}\Out(G)\to 1.$$

The separability version of Grossman's criterion is:

\begin{prop}\label{prop:super_grossman}
Let $G$ be a finitely generated group with trivial centre and let $H<\Out(G)$. Let $\Gamma_H=\Psi^{-1}(H)$ and let $\phi\in \Aut(G)$. Suppose that:

\begin{enumerate}
\item\label{item:pointwise} there exists $x\in G$ such that $\phi(x)\neq h(x)$ for all $h\in\Gamma_H$.

\item $\Gamma_H$ is conjugacy separable.\label{item:conj_sep}
\end{enumerate}
Then there exists a finite group $F$ and a homomorphism $q:\Out(G)\to F$ with $q(\Psi(\phi))\notin q(H)$.
\end{prop}

\begin{remark}
Taking $H=\{1\}$ in the above proposition yields \cite[Theorem 1]{Grossman} in the case where $G$ has trivial centre.
\end{remark}

\begin{proof}[Proof of Proposition~\ref{prop:super_grossman}]
Item \eqref{item:pointwise} provides $x\in G$ such that $\phi(x)\notin \Gamma_H\cdot x$.

First observe that $i(\phi(x))$, which is contained in $\Gamma_H$, cannot be conjugate in $\Gamma_H$ to $i(x)$.  Indeed, if it were, then for some $\gamma\in\Gamma_H$, we would have $i(\phi(x))=\gamma i(x)\gamma^{-1}=i(\gamma(x))$, with the second equality following from Remark~\ref{rem:conjugation_in_extension}.  Injectivity of $i$ would then imply $\gamma(x)=\phi(x)$, a contradiction. 


Hence, by item \eqref{item:conj_sep}, there is a finite quotient $\pi:\Gamma_H\to F_0$ such that $\pi(i(\phi(x)))$ is not conjugate in $F_0$ to $\pi(i(x))$.  It follows that $\pi(i(\phi(x)))\not \in \pi(i(\Gamma_H\cdot x))$.  Indeed, by Remark~\ref{rem:conjugation_in_extension}, for all $\gamma\in\Gamma_H$, we have that $\pi(i(\gamma(x)))$ is conjugate in $F_0$ to $\pi(i(x))$, which is not conjugate to $\pi(i(\phi(x)))$.

Consider the quotient $\pi\circ i:G\to \bar G\leq F_0$.  Let $K$ be the intersection of all subgroups of $G$ with the same index as $\kernel(\pi\circ i)$. This is a characteristic subgroup, and furthermore, since $G$ is finitely generated, $K$ has finite index in $G$. Denote the corresponding quotient map $q'':G\to G/K$. Since $\pi\circ i$ factors through $q''$ we have $q''(\phi(x))\notin q''(\Gamma_H\cdot x)$. 

Since $K$ is characteristic, $q''$ induces a homomorphism $q':\Aut(G)\to\Aut(G/K)$ defined by $q'(\gamma)(gK) = q''(\gamma(g))$. This has the property that $q'(\phi)\notin q'(\Gamma_H)$, because $q'(\phi)(xK)\notin q'(\Gamma_H)(xK)$, since the former is $q''(\phi(x))$ and the latter is $q''(\Gamma_H\cdot x)$.

The homomorphism $q'$ descends to a homomorphism $q:\Out(G)\to\Out(G/K)$ fitting into the following commutative diagram:

\begin{center}
    $
    \begin{diagram}
    \node{\Aut(G)}\arrow{e,t}{\Psi}\arrow{s,l}{q'}\node{\Out(G)}\arrow{s,r}{q}\\
    \node{\Aut(G/K)}\arrow{e,b}{\Psi_K}\node{\Out(G/K)}\\
    \end{diagram}
    $
\end{center}

Since $q'(\phi)\notin q'(\Gamma_H)$ and $\Inn(G/K)=q'(\Inn(G))<q'(\Gamma_H)$, we have $\Psi_K(q'(\phi))\notin \Psi_K(q'(\Gamma_H))$. In view of the commutative diagram, the former is $q(\Psi(\phi))$ and the latter is $q(H)$, so that we found a finite quotient of $\Out(G)$ that separates $\Psi(\phi)$ from $H$, namely $F=\Out(G/K)$, which is finite since $K$ has finite index in $G$.
\end{proof}

\section{Mapping class group lemmas}\label{sec:MCG_lemmas}
In Section~\ref{subsec:finitely_many_non_pA}, we reduce condition \eqref{item:pointwise} from Proposition~\ref{prop:super_grossman} to a statement about finitely many automorphisms of a surface group; this is used in the proof of Theorem~\ref{thm:criterion}.  In Section~\ref{subsec:no_hidden_symmetries}, we prove a lemma needed for verifying malnormality of the examples from Theorem~\ref{thm:intro_examples}.

\subsection{Finitely many non-pseudo-Anosovs}\label{subsec:finitely_many_non_pA}
The goal of this subsection is to prove:

\begin{prop}\label{prop:non-pA}
Let $\Sigma$ be a finite-type surface and let $H<MCG(\Sigma)$ be a torsion-free malnormal convex-cocompact subgroup. Let $g\in MCG(\Sigma)-H$ have infinite order. Then the coset $gH$ contains finitely many elements that are not pseudo-Anosov.
\end{prop}

The proof is postponed until after some preparatory lemmas.  We will use the following lemma to check that specific elements are pseudo-Anosov. The lemma gives two alternative conditions, the first one is designed to show that the product of two elements one of which acts elliptically and has "large rotation angle" is loxodromic. The second property, instead, captures a product were one element is loxodromic with suffieicntly large translation length.

\begin{lem}[Loxodromic products]\label{lem:gh}
For every $\delta,Q\geq 0$ there exists a constant $C_{Q,\delta}$ with the following property. Let $A$ be a group acting on a $\delta$-hyperbolic space $X$. Let $g_0,g_1\in A$ be elements that respectively stabilise $Q$-quasi-convex sets $Y_0,Y_1\subset X$. Let $\pi_i$ be closest-point projection to $Y_i$. Assume that one of the following holds:
\begin{enumerate}
    \item let $y_i\in \pi_i(Y_{i+1})\subseteq Y_i$ and suppose that $d(Y_0,Y_1)\geq C_{Q,\delta}$ and $$\diam_X(N_{3\delta}([y_0,y_1])\cap g_i[y_0,y_1])\leq \frac{C_{Q,\delta}}{3}$$ for $i=0,1$; or\label{item:far-apart}
    \item $diam(\pi_1(g_0Y_1))\leq C_{Q,\delta}/3$ and $d(g_1y,y)\geq C_{Q,\delta}+2d(y,g_0y)$ for some $y\in Y_1$.\label{item:long-translation}
\end{enumerate}
Then $g_0g_1$ is loxodromic.
\end{lem}

\begin{proof}
In case \eqref{item:far-apart} we can form a bi-infinite quasi-geodesic on which $g_0g_1$ acts by translation by concatenating suitable translates of $U=[y_1,y_0]$, $V=[y_0,g_0y_0]$, and $W=[y_1,g_1y_1]$, with every other geodesic being a translate of $[y_0,y_1]$.  Specifically, the $\langle g_0g_1\rangle$--translates of the concatenation 
$$U\cdot V\cdot (g_0\bar U)\cdot (g_0W),$$
which joins $y_1$ to $g_0g_1y_1$,
concatenate to form a bi-infinite path, shown in Figure~\ref{fig:gh}.  Here the overline denotes the reverse path. In the notation of Lemma \ref{lem:concat}, by our assumption we have $d_{i,i\pm 2}\leq C_{Q,\delta}/3$, while the $d_{i,i\pm 1}$ are bounded independently of $C_{Q,\delta}$ since the $[y_0,y_1]$ are coarsely shortest geodesics connecting quasiconvex sets.  So Lemma~\ref{lem:concat} implies $g_0g_1$ is loxodromic provided $C_{Q,\delta}$ is sufficiently large in terms of $Q$ and the constant $C_\delta$ from the lemma, which depends only on $\delta$.

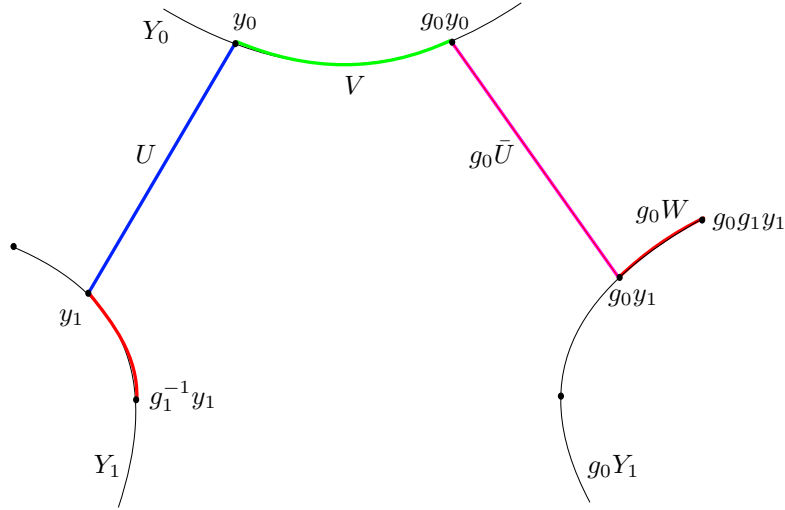
\begin{figure}[h]
\begin{overpic}[width=0.6\textwidth]{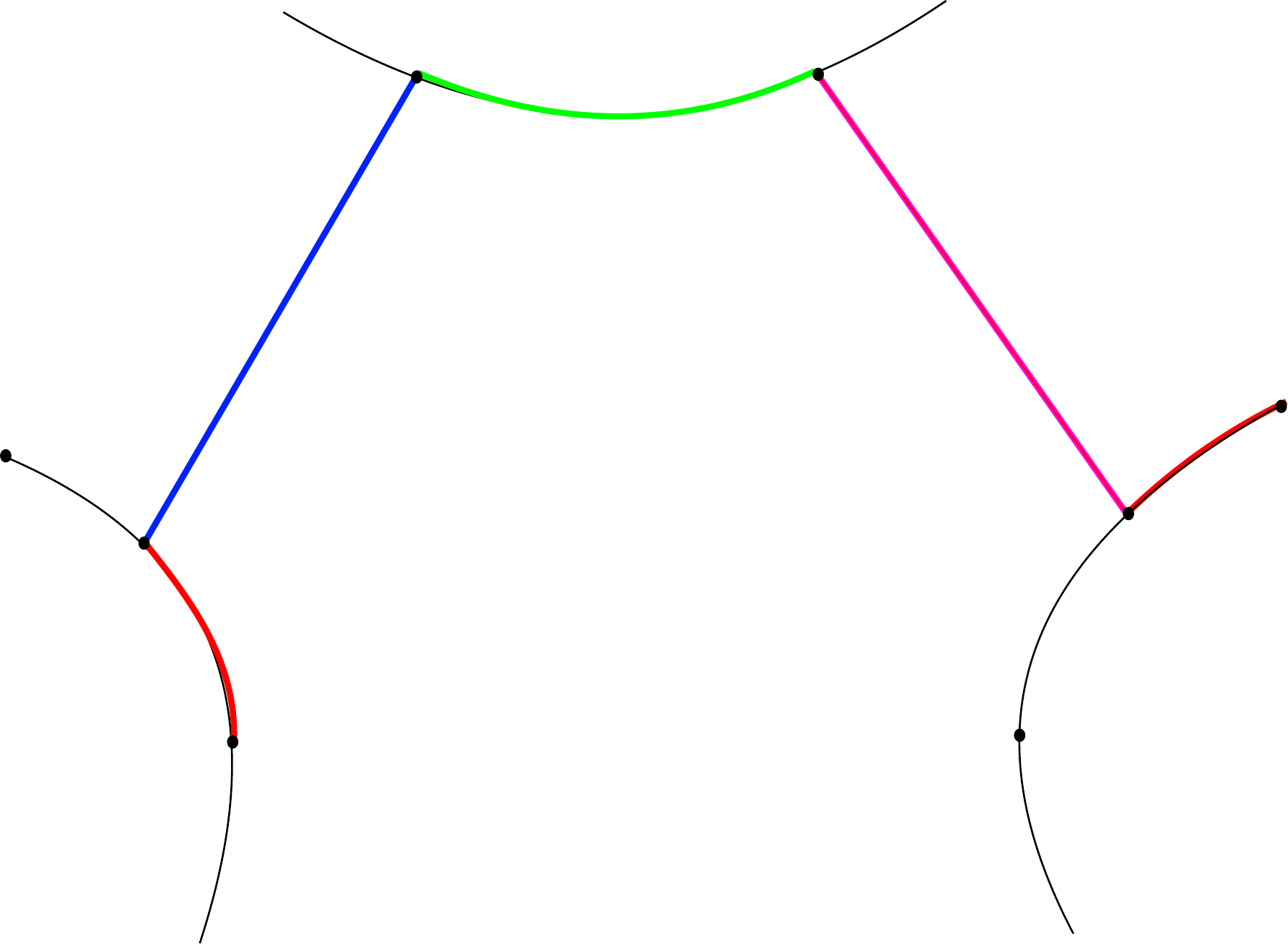}
\put(12,5){$Y_1$}
\put(19,68){$Y_0$}
\put(83,5){$g_0Y_1$}
\put(18,50){$U$}
\put(48,60){$V$}
\put(66,50){$g_0\bar U$}
\put(90,42){$g_0W$}
\put(7,27){$y_1$}
\put(32,70){$y_0$}
\put(59,70){$g_0y_0$}
\put(86,30){$g_0y_1$}
\put(20,15){$g_1^{-1}y_1$}
\put(101,41){$g_0g_1y_1$}
\end{overpic}
\caption{Proof of Lemma~\ref{lem:gh}.\eqref{item:far-apart}.  The paths concatenated to form the $g_0g_1$--axis are coloured according to their $\langle g_0g_1\rangle$--orbits.}\label{fig:gh}
\end{figure}

Case \eqref{item:long-translation} is similar: we concatenate translates of $U=[y,g_1y]$ and $V=[y,g_0y]$.  Specifically, the $\langle g_0g_1\rangle$--translates of the concatenation $V\cdot g_0U$
concatenate to form a bi-infinite path.  Again, our assumptions allow us to apply Lemma \ref{lem:concat}.  Indeed, successive translates of $U$ have $d_{i\pm 2}$ bounded since intersections of neighborhoods of the relevant translates of $Y_1$ have diameter bounded in terms of the diameter of $\pi_1(g_0Y_1)$, so the assumption on $d(y,g_1y)$, which controls $d_{i\pm1}$, shows that Lemma~\ref{lem:concat} applies (with the translates of $U$ as the even-indexed $\alpha_i$) provided $C_{Q,\delta}$ is large enough in terms of $Q$ and $\delta$.
\end{proof}

We will also need the following lemmas to check condition \ref{item:far-apart} in Lemma~\ref{lem:gh} when using it.  Recall that for curve graphs we take $\delta=100$.

The following lemma can be formulated more generally for a group acting acylindrically on a hyperbolic space and a infinite-order element whose orbits are bounded, but we prefer to state just the special case we need to avoid discussing coarse fixed point sets.

\begin{lem}\label{lem:simplex_translate}
Let $\Sigma$ be a finite-type surface. Then there exists $D\geq 0$ with the following property. Let $g\in MCG(\Sigma)$ be an infinite order reducible element, fixing the clique $Y$ in $\mathcal C(\Sigma)$. Then for any geodesic $\gamma$ starting in $Y$ we have that $N_{300}(\gamma)\cap g\gamma$ has diameter at most $D$.
\end{lem}

\begin{proof}
By hyperbolicity, for every $L\geq 0$ there exists $D\geq 0$ such that the following holds for $g,\gamma$ as in the statement: if $p\in\gamma$ is moved distance at most $L$ by a power of $g$, then $\gamma$ has an initial subgeodesic of length $d(p,Y)-D$ that gets moved distance at most $300$ by the same power of $g$. Therefore, if the lemma did not hold for a sufficiently large $D$, we would get a contradiction with acylindricity of the $MCG(\Sigma)$--action on $\mathcal C(\Sigma)$.
\end{proof}

\begin{lem}\label{lem:geodesic-translate}
Let the group $A$ act acylindrically on the hyperbolic space $X$. Then there exists a constant $D$ such that the following holds. Let $g\in A$ be loxodromic and let $Y$ be its quasi-axis as in Proposition \ref{prop:E}.\eqref{item:quasiline}. Then for any $x\in X$ and geodesic $\gamma=[x,\pi_Y(x)]$ we have that $N_{3\delta}(\gamma)\cap g\gamma$ has diameter at most $D$.
\end{lem}

\begin{proof}
This is similar to Lemma~\ref{lem:simplex_translate}, except we "shorten" $\gamma$ on both sides.
\end{proof}



We are ready for:

\begin{proof}[Proof of Proposition \ref{prop:non-pA}]
Let $\Sigma,H,g$ be as in the statement and fix an $H$--orbit $H\cdot o$ in $\mathcal C(S)$.  Let $Q$ be the constant from Proposition~\ref{prop:E}.  For each nontrivial $h\in H$, let $Y_h$ be the $Q$--quasiconvex quasi-axis for $h$, from Proposition~\ref{prop:E}.\eqref{item:quasiline} (any nontrivial $h\in H$ is pseudo-Anosov). Let $Y_g\subset\mathcal C(\Sigma)$ be the following $Q$--quasiconvex subspace:
\begin{itemize}
    \item if $g$ is reducible, then $Y_g$ is a clique stabilised by $g$;
    \item if $g$ is pseudo-Anosov, then $Y_g$ is the quasi-axis from Proposition~\ref{prop:E}.\eqref{item:quasiline}.
\end{itemize}

Let $D$ be the maximum of the constants from Lemmas \ref{lem:simplex_translate} and \ref{lem:geodesic-translate}.  Let $C_0$ be such that $Y_h\subset N_{C_0}(H\cdot o)$ for all $h\in H$ (which exists since $Y_h$ and $H\cdot o$ are uniformly quasiconvex  and $Y_h$ lies in some neighborhood of $H\cdot o$).  By enlarging $Q$ once, we can assume $d(y',gy')\leq Q$ for all $y'\in Y_g$.

We will need the following claim:

\begin{claim}\label{claim:malnormality_consequence}
For any $L\geq 0$, the set $Y_g\cap N_{L}(H\cdot o)$ is bounded.  Moreover, there exists $C_1\geq 0$ such that $\pi_{Y_h}(gY_h)$ has diameter at most $C_1$ for all $h\in H$.
\end{claim}

\begin{proof}
If $Y_g\cap N_{L}(H\cdot o)$ had unbounded diameter then $gH\cdot o$ would have unbounded intersection with the $2L+Q$--neighborhood of $H\cdot o$, since $Q\geq d(y',gy')$ for all $y'\in Y_g$. Lemma \ref{lem:??} then gives that $gH g^{-1}\cap H$ is infinite, contradicting malnormality of $H$.  

Similarly, malnormality of $H$ and Lemma~\ref{lem:??} imply that for all $L\geq 0$ there exists $B(L)$ such that $Y_h\cap N_{L}(gY_h)$ has diameter at most $B(L)$ (since $Y_h\subset N_{C_0}(H\cdot o)$ and $gY_h\subset N_{C_0}(gH\cdot o)$ and the lemma bounds the diameter of coarse intersections of the two coset orbits). It is well-known that quasiconvex sets in a hyperbolic space that have bounded intersection of neighborhoods also have bounded projection onto each other; we sketch this argument in our notation. 

Consider two points $a,b\in gY_h$ and $a'\in\pi_{Y_h}(a),b'\in\pi_{Y_h}(b)$. Any point on $[a',b']$ which is further away than $500+2Q$ from $\{a',b'\}$ cannot be $200$-close to $[a,a']$ or $[b,b']$, otherwise $a'$ would not be nearly closest to $a$ (that is, $a'\notin \pi_{Y_h}(a)$) or $b'$ would not be nearly closest to $b$. Hence, any such point is within $200$ of $[a,b]$ by thinness of quadrilaterals. Hence such point is within $200+Q$ from $gY_h$, as well as within $Q$ of $Y_h$. Hence, the diameter of $Y_h\cap N_{L}(gY_h)$ for $L=200+2Q$ is at least $d(a',b')-10^3-10Q$. Since $a',b'$ were arbitrary, the diameter of $\pi_{Y_h}(gY_h)$ is bounded in terms of $B(L)$.
\end{proof}

Finally, let $C_{Q,100}$ be the constant from Lemma~\ref{lem:gh}.  We can assume that $C_{Q,100}\geq 10D+10^3+3C_1$.

Fix $h\in H$.  First suppose that $d(Y_g,Y_h)>C_{Q,100}$.  Then Lemma \ref{lem:gh}.\eqref{item:far-apart} applies in view of Lemma \ref{lem:geodesic-translate} (if $g$ is pseudo-Anosov), or Lemma~\ref{lem:simplex_translate} (if $g$ is infinite-order reducible). Hence Lemma~\ref{lem:gh}.\eqref{item:far-apart} shows that $gh$ is pseudo-Anosov. 
    
Next, suppose that the translation length of $h$ satisfies $\tau_h>2Q+10C_{Q,100}$ and that $d(Y_g,Y_h)\leq C_{Q,100}$.  Recall that $diam(\pi_{Y_h}(gY_h))\leq C_1\leq  C_{Q,100}/3$.  Choose $y\in Y_h$ within distance $C_{Q,100}$ of some $y'\in Y_g$.
Hence $$2d(y,gy)+C_{Q,100}\leq 2d(y',gy')+5C_{Q,100}<\tau_h\leq d(y,hy).$$  So we can apply Lemma~\ref{lem:gh}.\eqref{item:long-translation} (with $g_1=h$ and $g_0=g$) and see that $gh$ is pseudo-Anosov.

To conclude, it suffices to prove the following:

\begin{claim}\label{claim:short}
There exist finitely many $h$ such that $d(Y_g,Y_h)\leq C_{Q,100}$ and $\tau_h\leq 2Q+10C_{Q,100}$.
\end{claim}

\begin{proof}
By Claim~\ref{claim:malnormality_consequence} (which says that $Y_g$ can only come close to $H\cdot o$ at a bounded set) and since $Y_h$ lies within $C_0$ of $H\cdot o$, we can only have $d(Y_g,Y_h)\leq C_{Q,100}$ if $d(Y_h,o)\leq C'$ for some $C'$ (depending on $g$ also).

Since $H$ is quasi-isometrically embedded in $\mathcal C(\Sigma)$, it suffices to show that in (a fixed Cayley graph of) a hyperbolic group $G$, for any $B\geq 0$ there are only finitely many $h\in G$ with translation distance bounded by $B$ and such that $Y_h$ contains a point within $B$ of the identity (note that quasi-axes in $G$ get mapped within bounded distance of quasi-axes in $\mathcal C(\Sigma)$, and translation lengths get distorted in a way controlled by the quasi-isometry constants). This is true since any such $h$ moves the identity a uniformly bounded amount.
\end{proof}

We have shown that for all $h\in H$, either $gh$ is pseudo-Anosov or $h$ is one of the finitely many elements of $H$ from Claim~\ref{claim:short}, which proves the proposition.
%
%
%
\end{proof}

\subsection{Pseudo-Anosovs without hidden symmetries}\label{subsec:no_hidden_symmetries}

\begin{lem}\label{lem:no_hidden_sym}
Let $\Sigma$ be a finite-type hyperbolic surface with genus $\mathfrak g$ and $\mathfrak p$ punctures, with $(\mathfrak g,\mathfrak p)\notin \{(1,1),(1,2),(2,0)\}$. Then there exists a pseudo-Anosov $\phi\in MCG(\Sigma)$ and $n_0\geq 0$ such that the following holds. For any curve $c$ on $\Sigma$ and $n\geq n_0$ we have $\stab(c)\cap \stab(\phi^n(c))=\{1\}$.
\end{lem}

\begin{proof}
Proposition \ref{prop:E}.\eqref{item:coarse_stab} yields, for a pseudo-Anosov $g$, a unique maximal elementary subgroup $E(g)$ containing $g$ and an invariant quasi-geodesic $\gamma_g$ in the curve graph $\mathcal C(\Sigma)$, with the following property. For every $D\geq 0$ there exists $L\geq 0$ such that if $h\in G$ has the property that there exist $x,y\in \gamma_g$ with $d_{\mathcal C(\Sigma)}(x,y)\geq L$ and $d_{\mathcal C(\Sigma)}(x,h(x)),d_{\mathcal C(\Sigma)}(y,h(y))\leq C$, then $h\in E(g)$.

Let $\phi$ be a pseudo-Anosov for which $E(\phi)=\langle \phi\rangle$.  Such a $\phi$ exists by \cite[Theorem 6.14]{DGO}, since the maximal finite normal subgroup of a mapping class group coincides with its centre (see the proof of \cite[Proposition 11.15]{MT:Cremona}) and the centre is trivial under our assumptions on $(\mathfrak g,\mathfrak p)$ (see \cite[Section 3.4]{FarbMargalit}).

Given any curve $c$ and any sufficiently large $n$, if $h\in \stab(c)\cap \stab(\phi^n(c))$, then $h$ coarsely stabilises the closest-point projections of $c$ and $\phi^n(c)$ to $\gamma_\phi$, so that $h\in E(\phi)=\langle \phi\rangle$ by the above. No non-trivial power of $\phi$ fixes a curve, so $h=1$.
\end{proof}


\section{Common witnesses for non-inner automorphisms}

\begin{prop}\label{prop:not-conjugate}
Let $G$ be the fundamental group of a closed connected orientable surface of finite genus, and let $\phi_1,\dots,\phi_n$ be non-inner automorphisms of $G$. Then there exists $x\in G$ such that for all $i$, the elements $\phi_i(x)$ and $x$ are not conjugate.
\end{prop}

We need the following version of \cite[Lemma 4.4]{MinasyanOsin:normal} (which we use in the proof).

\begin{lem}\label{lem:cyclicperm}
Let $G$ be a torsion-free hyperbolic group. Let $l\geq 2$ and let $\{g_1,\dots, g_l\}$, $\{h_1,\dots, h_l\}$ be sets of non-commensurable, primitive, infinite order elements. Assume that if $g_1^{m_1}\dots g_l^{m_l}$ is the identity of $G$ for some $m_i$ then all $m_i$ are $0$.

There exists  $N \in \mathbb N$ such that the following holds. Suppose that
$$g=g_1^{n_1}g_2^{n_2}\dots g_l^{n_l}$$ is conjugate to $$h=h_1^{n_1} \cdots h_l^{n_l}$$ in $G$, and for all $1\leq i\leq l$ we have $n_i\geq N$.  Then there exists $x\in G$ and $k\in\mathbb N$ with $h_i=x g_{i+k}x^{-1}$ for all $i$ (indices mod $l$).
\end{lem}

\begin{proof}
We conflate $G$ with a fixed Cayley graph.


There exists a constant $C_0$ such that for all sufficiently large $N$ we have that $g$ stabilises a $(C_0,C_0)$--quasigeodesic $\alpha_g$ obtained by concatenating translates of the discrete paths $\{1,g_i,g_i^2,\dots, g_i^{n_i}\}$, with $n_i\geq N$, that we call \emph{pieces}, and similarly for $h$. This is because the $\langle g_i\rangle$ are quasi-geodesic lines, and moreover any two translates of two distinct such quasi-geodesic lines cannot stay within bounded distance of each other for arbitrarily long time in view of Proposition \ref{prop:E}.\eqref{item:coarse_stab} ($g_i$ being not commensurable to $g_j$ is equivalent to $g_i\notin E(g_j)$).

Let $D\geq 0$ be such that any two $(C_0,C_0)$--quasigeodesics in $G$ with the same endpoints at infinity lie within distance $D$ of each other.

If $g$ and $h$ are conjugate, we have that two translates of $\alpha_g$ and $\alpha_h$ have the same endpoints at infinity. Hence, for all $M$ the following holds for all large $N$: each piece for $h$ has a subpath of length at least $M$ that has a translate that lies within $D$ of a piece for $\alpha_g$.  See Figure \ref{fig:pieces}.
\medskip

\begin{figure}[h]
\begin{overpic}[width=\textwidth]{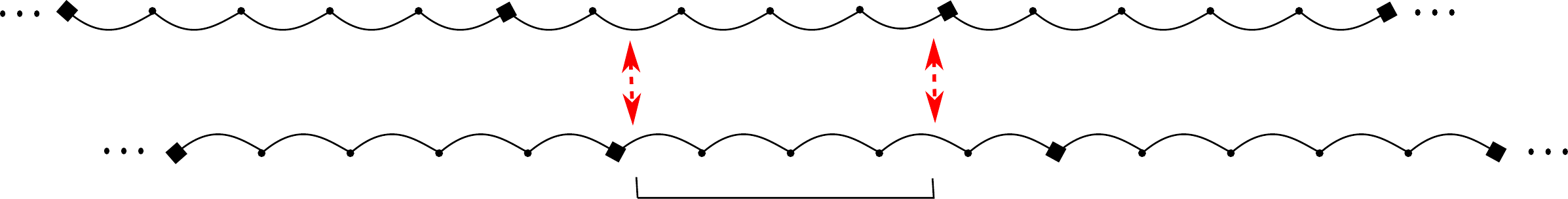}
\put(62,6){$D$}
\put(36,6){$D$}
\put(49,-2){$M$}
\end{overpic}
\caption{Edges in the bottom axis $\alpha_h$ are labelled by the $h_i$, with a translate of the piece $\{1,h_i,\ldots,h_i^5\}$ shown in the middle.  This has a long subpath $D$--close to a subpath of the piece $\{1,g_{\sigma(i)},\ldots,g_{\sigma(i)}^5\}$ of the translate of $\alpha_g$ shown at the top. (Large dots represent concatenation points of successive pieces.)}\label{fig:pieces}
\end{figure}

For $M$ large enough, this implies that each $h_i\in x_i E(g_{\sigma(i)})x^{-1}_i$ for some $\sigma(i)\in\{1,\ldots,l\}$ and some $x_i\in G$ by Proposition \ref{prop:E}.\eqref{item:coarse_stab}. Note that $i\mapsto \sigma(i)$ is injective since $\{h_1,\ldots,h_l\}$ are pairwise non-commensurable.

To sum up, either there is no bijection $\sigma$ and elements $x_i$ such that $h_i$ lies in $x_i E(g_{\sigma(i)})x^{-1}_i$, in which case $g$ and $h$ cannot be conjugate for any large $N$ and the conclusion holds vacuously, or there exist $\sigma$ and $x_i$, in which case we can fix them and they only depend on the choice of $g_i$ and $h_i$.

We now fix $\sigma$ and $x_i$ as above. We further let $F$ be the (finite) collection of all elements of the form $x_i^{-1}x_{i+1}$ (indices mod $l$). Therefore, $h$ (whence $g$) is conjugate to $(h'_1)^{n_1}f_1 \cdots (h'_l)^{n_l}f_l$, for some $h'_i\in E(g_{\sigma(i)})$, where $f_i=x_i^{-1}x_{i+1}$ (indices mod $l$). We have now reduced to the setup of \cite[Lemma 4.4]{MinasyanOsin:normal}, from which we can conclude that $\sigma$ is a cyclic shift, say $\sigma(i)=i+k$ for some fixed $k$, and $f_i\in E(g_{\sigma(i)})E(g_{\sigma(i+1)})$.

Now, we have $E(g_j)=\langle g_j\rangle$ for all $j$ by the primitivity hypothesis and the assumption that $G$ is torsion-free, so $f_i=g_{\sigma(i)}^{l_i}g_{\sigma(i+1)}^{m_i}$ for some $l_i,m_i$. Up to multiplying each $x_i$ on the right by $g_{\sigma(i)}^{l_i}$ (which does not affect the equation $h_i=x_ih'_ix_i^{-1}$), we then get $x_{i+1}=x_ig_{\sigma(i+1)}^{m'_i}$, where $m'_i=m_i+l_{i+1}$. This holds with indices mod $l$, so that $x_1=x_1g_{\sigma(2)}^{m'_1}\dots g_{\sigma(l)}^{m'_{l-1}}g_{\sigma(1)}^{m'_l}$. By our hypothesis on the $g_i$ (applied to a cyclic conjugate) we get that all $m_i$ are $0$, and therefore we get that all $x_i$ are equal, giving us the required $x$.
\end{proof}

With the lemma in hand, we proceed to:

\begin{proof}[Proof of Proposition \ref{prop:not-conjugate}]
Fix some $\phi=\phi_i$, and let $g_1,\dots,g_l$ be the standard generators of the surface group $G$. Observe that the condition from Lemma \ref{lem:cyclicperm} about products of the form $g_1^{m_1}\dots g_l^{m_l}$ holds in this setup. Indeed, the $g_i$ are generators in a $C'(1/6)$ small cancellation presentation where the (only) relator $R$, as a cyclic word, has the following property. Any subword of a cyclic permutation of $R$ which has length more than half the length of $R$ is not a subword of a cyclic permutation of a product $g_1^{m_1}\dots g_l^{m_l}$. We can conclude in view of Greendlinger's lemma.

Let $h_j=\phi(g_j)$. Consider the elements $g=g_1^{T}g_2^{T+1}\dots g_l^{T+l-1}$ and $h=h_1^{T}h_2^{T+1}\dots h_l^{T+l-1}$, for $T$ larger than the $N$ from Lemma \ref{lem:cyclicperm}. We claim that $g$ and $h$ cannot be conjugate, which implies that we can take $x=g$ for some sufficiently large $T$ which works simultaneously for all $\phi_i$.

If $g$ and $h$ were conjugate, then we would have $h_i=xg_{i+k}x^{-1}$ for some fixed $x\in G$ and natural number $k<l$. We claim that $k=0$. This holds since there is a homomorphism $\psi$ to $\mathbb Z$ such that $\psi(g_1)=1$ and $\psi(g_j)=0$ for all $j\neq 1$, so that $\psi(g)=T$ and $\psi(h)=T+k$, so $k=0$. Hence $\phi(g_i)$ acts as conjugation by $x$ on all generators, implying that it is inner, contradicting the hypotheses.
\end{proof}

Now we can prove the separability criterion.

\begin{proof}[Proof of Theorem \ref{thm:criterion}]
Throughout the proof we think of the mapping class group as an index 2 subgroup of $\Out(\pi_1(\Sigma_{\mathfrak g}))$.  Let $\Psi:\Aut(\pi_1(\Sigma_{\mathfrak g}))\to \Out(\pi_1(\Sigma_{\mathfrak g}))$ be as in Section~\ref{subsec:grossman}.

Since $MCG(\Sigma_{\mathfrak g})$ is residually finite \cite{Grossman} and has finitely many conjugacy classes of finite-order elements \cite{Bridson}, there is a finite-index torsion-free subgroup $M_0\leq MCG(\Sigma_{\mathfrak g})$.

\begin{claim}\label{claim:inf_order_separate}
Suppose that $\phi\in \Aut(\pi_1(\Sigma_{\mathfrak g}))$ satisfies $\bar\phi=\Psi(\phi)\in M_0-(H\cap M_0)$.  Then there is a finite group $F$ and a quotient $q:\Out(\pi_1(\Sigma_{\mathfrak g}))\to F$ such that $q(\bar \phi)\notin q(H)$.  
\end{claim}

\begin{proof}[Proof of Claim~\ref{claim:inf_order_separate}]
We will use Proposition~\ref{prop:super_grossman}.  First, by hypothesis, $\Gamma_H=\Psi^{-1}(H)$ is conjugacy separable, so condition \eqref{item:conj_sep} of Proposition~\ref{prop:super_grossman} is satisfied.  So, we are left to verify condition~\eqref{item:pointwise} of Proposition~\ref{prop:super_grossman} for the given $\phi$.

Since $\bar\phi\in M_0-(H\cap M_0)$, the element $\bar \phi$ has infinite order in $MCG(\Sigma_{\mathfrak g})$, so Proposition~\ref{prop:non-pA} applies.  The only properties of $\phi$ used in the following argument are that $\bar\phi\in MCG(\Sigma)$, and $\bar \phi$ has infinite order, and $\bar \phi\not\in H$.  

We have to find $x\in \pi_1(\Sigma_{\mathfrak g})$ such that $\phi(x)\neq h(x)$ for $h\in \Gamma_H$. This is equivalent to $\phi^{-1}h(x)\neq x$.

If $\Psi(\phi^{-1}h)$ (which is necessarily in $MCG(\Sigma_{\mathfrak g})$ since $\Psi(h)\in H$ and $\bar \phi\in MCG(\Sigma_{\mathfrak g})$) is pseudo-Anosov, then any nontrivial $x$ satisfies the required condition for the $h$ in question. Hence, we are only concerned with the non-pseudo-Anosov elements of $\bar\phi^{-1}H$, of which there are only finitely many by Proposition \ref{prop:non-pA}. Taking arbitrary preimages of these elements, we obtain finitely many $h_i\in \phi^{-1}\Gamma_H$ such that if $x\in G$ has the property that $h_i(x)$ is not conjugate to $x$ for all $i$, then $x$ satisfies the required property, that is, $\phi(x)\neq h(x)$ for all $h\in \Gamma_H$.


The existence of such $x$ is guaranteed by Proposition \ref{prop:not-conjugate}.  Hence Proposition~\ref{prop:super_grossman} supplies a finite quotient $q:\Out(\pi_1(\Sigma_{\mathfrak g}))\to F$ such that $q(\bar \phi)\not\in q(H)$.
\end{proof}

From Claim~\ref{claim:inf_order_separate}, we get that $H\cap M_0$ is separable in $M_0$, since for each $\phi\in M_0-H\cap M_0$, we can restrict the quotient $q$ from the claim to $M_0\leq \Out(\pi_1(\Sigma_{\mathfrak g}))$ to get a finite quotient $q:M_0\to \bar M_0\leq F$ where $q(\phi)\not\in q(H\cap M_0)$. In summary, we have a finite-index subgroup $M_0\leq MCG(\Sigma_{\mathfrak g})$ and a subgroup $H\leq MCG(\Sigma_{\mathfrak g})$ such that $H\cap M_0$ is separable in $M_0$.  Applying Proposition 2.2.(ii) of \cite{Niblo} (taking $K=K_0=\{1\}$ in the statement of that proposition) shows that $H$ is separable in $MCG(\Sigma_{\mathfrak g})$, as required.
%
\end{proof}

\section{MMS subgroups}

We finally construct examples to which Theorem~\ref{thm:criterion} applies.  We will use the following form of the Bounded Geodesic Image Theorem. We denote the subsurface projection to the subsurface $Y$ by $\pi_Y$ and abbreviate $d_{\mathcal C(Y)}(\pi_Y(\cdot),\pi_Y(\cdot))=d_{\mathcal C(Y)}(\cdot,\cdot)$.

\begin{thm}\cite[Theorem 3.1]{MM2}\label{thm:BGI}
For each $\mathfrak g\geq 2$ there exists $B\geq 0$ such that the following holds. If the vertices $x,y\in \mathcal C(\Sigma)$ satisfy $d_{\mathcal C(\Sigma_{\mathfrak g}-v)}(x,y)\geq B$ for some non-separating curve $v$, then any geodesic from $x$ to $y$ contains $v$.
\end{thm}

We now prove Theorem~\ref{thm:intro_examples}:

\begin{thm}\label{thm:examples}
For every $\mathfrak g\geq 3$ and $n\geq 1$ there exist infinitely many conjugacy classes of malnormal convex cocompact subgroups $H$ of $MCG(\Sigma_{\mathfrak g})$ isomorphic to a free group of rank $n$ whose corresponding extension $\Gamma_H$ is virtually compact special.
\end{thm}

\begin{proof}
In \cite{MMS} the authors construct subgroups (that we call \emph{MMS subgroups}) satisfying all requirements except possibly malnormality.  MMS subgroups are described in \cite[Proposition 5.7, Theorem 5.8]{MMS}.  Our goal is to construct (many) malnormal MMS subgroups. 

Each MMS subgroup $H$ stabilises a \emph{tight tree of homologous non-separating curves} (a \emph{tight tree} for short), which is a tree $T$ with the following description (see \cite[Section 3.1]{MMS}):  

\begin{itemize}
    \item For each vertex $v$ of $T$ we have a vertex $c_v\in \mathcal C(\Sigma_{\mathfrak g})$, and for each edge $e$ of $T$, we have a tight geodesic $\gamma_e$ in $\mathcal C(\Sigma_{\mathfrak g})$ connecting the vertices corresponding to the endpoints of $e$.
    
    \item All multicurves occurring along $\gamma_e$ (including the endpoints) are formed from non-separating curves.  Moreover, all curves appearing along all tight geodesics are homologous to each other.
\end{itemize}

There are some additional properties of $T$ guaranteeing that the natural subdivision of $T$ isometrically embeds into $\mathcal C(\Sigma_{\mathfrak g})$ (\cite[Proposition 3.2]{MMS}).  Moreover, the tight tree $T$ stabilised by an MMS subgroup $H$ has the property that whenever $u,v,w$ are consecutive vertices along a tight geodesic as above, then there is a pseudo-Anosov on $\Sigma_{\mathfrak g}-\{v\}$ which maps $u$ to $w$. In fact, the tree is constructed starting from a choice of such partial pseudo-Anosovs and any choice of sufficiently large powers of the partial pseudo-Anosovs (\cite[Section 5.1]{MMS}).

As a result, there are MMS subgroups $H$ that stabilise a tight tree $T$ with the following additional properties. For an edge $e$ of the tree, we call the \emph{projection numbers} the quantities $d_{\mathcal C(\Sigma_{\mathfrak g}-v)}(u,w)$ for consecutive $u,v,w$ along $\gamma_e$. Then
\begin{itemize}
    \item for any edge $e$ of $T$, all projection numbers are different and larger than $10B$, for $B$ as in Theorem~\ref{thm:BGI}, 
    \item for any two edges $e,f$ of $T$ belonging to different $H$-orbits, all projection numbers for $e$ are different from all projection numbers for $f$,
    \item the tight geodesics all have length at least $10^3$,
    \item $Stab_{MCG(\Sigma_{\mathfrak g}-v)}(u)\cap Stab_{MCG(\Sigma_{\mathfrak g}-v)}(w)$ is trivial for all $u,v,w$ that are consecutive in $\gamma_e$, where $e$ is any edge of $T$.
\end{itemize}

The last item can be arranged using the pseudo-Anosovs constructed in Lemma \ref{lem:no_hidden_sym}. 


We now verify malnormality of $H$.  Suppose for some $g\in MCG(\Sigma_{\mathfrak g})$ that $gHg^{-1}\cap H$ is infinite.  Then $T$ and $gT$ respectively contain bi-infinite geodesics $\alpha$ and $g\alpha'$ that 200-fellow-travel.

Fix an edge $e$ of $T$ such that $\gamma_e\subset\alpha$. We claim that for all vertices $c\in \gamma_e$ that are at distance more than $200$ from both endpoints of $\gamma_e$, the geodesic $g\alpha'$ passes through $c$. Notice that then the third item provides at least 100 such vertices $c$.

To prove the claim we use the lower bound on projection numbers from the first of the above four items and Theorem~\ref{thm:BGI}, as follows. First of all, by Theorem~\ref{thm:BGI} the endpoints $a',b'$ of $\gamma_e$ each project to $\mathcal C(\Sigma_{\mathfrak g}-c)$ within $B$ of the projection of one of the vertices $a,b$ adjacent to $c$ along $\gamma_e$. Again by Theorem \ref{thm:BGI} points $a'',b''$ in $g\alpha'$ that lie within $200$ of $a',b'$ project within $2B$ of the projections of $a,b$. In particular, $a'',b''$ project $B$-far on $C(\Sigma_{\mathfrak g}-c)$, and once again by Theorem \ref{thm:BGI} we have that $g\alpha'$ passes through $c$.

With another application of the third item, there exists an edge $f$ of $T$ such that there are distinct triples $u,v,w$ and $u_1,v_1,w_1$ of $\gamma_e$--consecutive vertices, both of which are contained in $g\gamma_f$.

By the second item, $e$ and $f$ are in the same $H$--orbit.  Hence there exists $h\in H$ such that $u,v,w,u_1,v_1,w_1\in \gamma_e\cap (gh)\gamma_e$.  

By the "projection numbers are different" clause of the first item, $ghu=u, ghv=v,ghw=w$.  From the fourth item, it follows that $gh$ is a power of the Dehn twist around $v$.  

The same argument applied to $u_1,v_1,w_1$ shows that $gh$ is also a power of the Dehn twist around $v_1$, so $gh=1$ and thus $g\in H$, i.e. $H$ is malnormal.


To arrange infinitely many conjugacy classes the argument is similar, arranging the projection numbers to be different for tight trees associated to different subgroups.
\end{proof}




\bibliographystyle{alpha}
\bibliography{biblio.bib}

\end{document}